\documentclass[12pt]{article}
\usepackage{amsmath,amssymb,amsthm}
\usepackage{enumerate}
\usepackage{color}
\usepackage{authblk}

\newtheorem{thm}{Theorem}[section]
\newtheorem{cor}[thm]{Corollary}
\newtheorem{lem}[thm]{Lemma}

\theoremstyle{definition}
\newtheorem{defn}[thm]{Definition}
\newtheorem{example}[thm]{Example}
\newtheorem{rem}[thm]{Remark}

\theoremstyle{remark}

\numberwithin{equation}{section}

\newcommand{\bp}{{\mathbf p}}
\newcommand{\bq}{{\mathbf q}}

\begin{document}

\title{Ergodicity for $p$-adic continued fraction algorithms}

\author{Hui Rao and Shin-ichi Yasutomi$^\dag$}

\maketitle

\footnote[0]{$\dag$ The correspondence author.}

\begin{abstract}
Following Schweiger's  generalization of multidimensional continued fraction algorithms, we consider a very large family of
$p$-adic multidimensional continued fraction algorithms, which include Schneider's algorithm,
Ruban's algorithms, and the $p$-adic Jacobi-Perron algorithm as special cases.
The main result  is to show that all the transformations in the family
are  ergodic with respect to the Haar measure.
\end{abstract}

\section{Introduction}
The classical continued fraction algorithm and its generalizations have been widely studied by many authors.
Schweiger (\cite{S}) provided a generalization of multidimensional continued fraction algorithms
as piecewise fractional linear maps, and studied its relation with dynamical system, ergodic theory and number theory.
As early as 1970's, there are several works on $p$-adic continued fraction algorithms, for example,
Schneider \cite{Sc} and Ruban \cite{R1}. Before stating these algorithms, let us introduce some notation first.

Let $p$ be a prime number and $\Bbb{Q}_p$(resp., $\Bbb{Z}_p$) be
the closure of  $\Bbb{Q}$(resp., $\Bbb{Z}$) with respect to  the $p$-adic topology.
For $x\in \Bbb{Z}_p$,  $ord_p(x)$ denotes the highest power of $p$ by which
$x$ is divided to be a $p$-adic integer.
It is well known that every $\alpha\in {\Bbb Q}_p\setminus\{0\}$ can be written as
\begin{align*}
\alpha=\sum_{n\in \Bbb{Z}}c_np^n,  \ \  c_n\in \{0,1,\ldots,p-1\},
\end{align*}
where $c_n=0$ for $-n$ sufficiently large.
We define the order and norm of $\alpha$ to be
$ord_p(\alpha):=\min\{n |c_n\ne 0\}$ and $|\alpha|_p:=p^{-ord_p(\alpha)}$
respectively, and set the residue class, integral part and fractional part of $\alpha$ to be
$$
\omega_p(\alpha):=c_0, \lfloor \alpha \rfloor_p:=\Sigma_{n\in \Bbb{Z}_{\leq 0}}c_np^n
\text{\ and\ }\langle \alpha \rangle_p:=\Sigma_{n\in \Bbb{Z}_{> 0}}c_np^n,
$$
respectively. By convention we set $ord_p(0)=\infty$ and $|0|_p=0$.

Schneider \cite{Sc}   introduced  the following $p$-adic continued fraction algorithm.
 Define $T_0:~p{\Bbb Z}_p\to p{\Bbb Z}_p$  as
\begin{align*}
T_0(x)=\dfrac{p^{ord_p(x)}}{x}-\omega_p\left(\dfrac{p^{ord_p(x)}}{x}\right).
\end{align*}
For $\xi\in p{\Bbb Z}_p$, denote  $\xi_n:=T_0^{n-1}(\xi)$  and
$a_n:=\omega_p\left(\dfrac{p^{ord_p(\xi_{n})}}{\xi_{n}}\right)$
for $n\in {\Bbb Z}_{>0}$.
Then  $\xi$ has the following $p$-adic continued fractional expansion (\cite{Sc})
\begin{align*}
\xi=\cfrac{p^{ord_p(\xi_{1})}}{a_1+\cfrac{p^{ord_p(\xi_{2})}}{a_2+\cfrac{p^{ord_p(\xi_{3})}}{a_3+\ldots}}}.
\end{align*}

Ruban \cite{R1}  considered  another algorithm given by
  the transformation $T_{\infty}: ~p{\Bbb Z}_p\to p{\Bbb Z}_p$    defined as
\begin{align*}
T_{\infty}(x)=\dfrac{1}{x}-\left\lfloor\dfrac{1}{x}\right\rfloor_p.
\end{align*}
Let $\xi\in p{\Bbb Z}_p$, $\xi_n=T_{\infty}^{n-1}(\xi)$ and
$a'_n:=\left\lfloor\dfrac{1}{\xi_{n}}\right\rfloor_p$
for $n\in {\Bbb Z}_{>0}$.
Then
\begin{align*}
\xi=\cfrac{1}{a'_1+\cfrac{1}{a'_2+\cfrac{1}{a'_3+\ldots}}}.
\end{align*}

It is well known that for the classical Gaussian map, the ergodic measure is $\ln(1+x)dx$.
In the $p$-adic case, Ruban \cite{R1} proved the ergodicity of $T_{\infty}$ and
Hirsh and Washington \cite{HW} confirmed the ergodicity of $T_{0}$, with respect to the Haar measure on $p{\Bbb Z_p}$.
Han\v{c}l, Ja\v{s}\v{s}ov\'a, Lertchoosakul and Nair \cite{HJLN} showed the mixing property of $T_{0}$.
Ruban \cite{R2,R3} also  proved the ergodicity of the $p$-adic version of Jacobi-Perron algorithm.

In this paper, following Schweiger's  generalization of multidimensional continued fraction algorithms, we consider a very large family of $p$-adic multidimensional continued fraction algorithms, which include Schneider's algorithm,  Ruban's algorithms, and the $p$-adic Jacobi-Perron algorithm as special cases (see Section 3).
The main result of the present paper is to show that all the transformations in the family
are  ergodic with respect to the Haar measure, and hence put the results of   \cite{R2,R3,HW}
in a uniform frame work.
Furthermore, we show that those transformations are mixing.

We note that there are other kind of generalizations of $p$-adic multidimensional continued fraction algorithms
which are not contained in our family,  for example, \cite{STY} and \cite{MT}.

The paper is organized as follows. In Section 2, we introduce $m$-dimensional $p$-adic linear fractional transformations
and discuss its basic properties. In Section 3, we introduce $m$-dimensional $p$-adic
continued fractional systems and we associate a transformation $T$ to each system; several families of examples are given.
In Section 4, we give several simple lemmas. The ergodicity of the transformation $T$ is proved in Section 5.

\section{Linear fractional transformation}
First, we review Schweiger's  generalization of multidimensional continued fraction algorithms  on $\mathbb R^n$ \cite{S}.
We say $\{A_\lambda\mid \lambda\in\Lambda\}$ is
a \emph{partition} of $A$, if $A=\bigcup_{\lambda\in \Lambda} A_\lambda$, and the union is disjoint.

\begin{defn}\label{Schweiger1} (\cite{S})
A pair $(B,T)$, where $B$ is a set and $T:B\to B$ is a map, is called {\it a fibred system} if
there exists a partition $\{B(i):i\in I\}$ of the set $B$, where $I$ is
finite  or countable, such that the restriction of $T$ to $B(i)$ is injective.
\end{defn}

\begin{defn}\label{Schweiger2} (\cite{S})
Let $n\in {\Bbb Z}_{>0}$.
A fibred system  $(B,T)$ is  called  a {\it piecewise fractional linear or a multidimensional} continued fraction if
\begin{enumerate}
\item $B$ is a subset of Euclidean space ${\Bbb R}^n$.
\item For every $k\in I$, there is an invertible matrix $\alpha_k=((A_{i,j}))$,
$i,j\in\{0,1,\dots, n\}$ such that  the restriction of $T$ to $B(k)$ can be written as
 $(y_1,\ldots,y_n)=T(x_1,\ldots,x_n)$   where
\begin{align*}
y_i=\dfrac{A_{i,0}+\sum_{j=1}^nA_{i,j}x_j}{A_{0,0}+\sum_{j=1}^nA_{0,j}x_j}.
\end{align*}
\end{enumerate}
\end{defn}

Among many others, Schweiger \cite{S} obtained the weak convergence of continued fractional expansion arising from the above algorithms, but it is difficult to show  ergodicity of the transformations $T$.

In this paper,  as an analogue of Schweiger \cite{S}, we introduce a large family of piecewise fractional linear maps in $p$-adic spaces; on contrary to the real case,  we will confirm the  ergodicity of such maps.
We define
\begin{equation}
\begin{array}{rl}
D_m:=\{&(x_1,\ldots,x_m) \in {\Bbb Q}_p^m\mid 
  1,x_1,\ldots,x_m \text{ are linearly independent over ${\Bbb Z}$} \}
\end{array}
\end{equation}
and
\begin{equation}
D_m':=D_m\cap   (p{\Bbb Z}_p)^m.
\end{equation}

\begin{rem} We remark that if $(x_1,\dots, x_m)\in D_m$, then  none of  $x_j$ can
  be a rational number, in particular, $x_j\neq 0$.
\end{rem}

We set $\mu$ to be  the Haar measure on ${\Bbb Q_p}$ (as an addition group), precisely, $\mu$ is  the Borel measure satisfying
$\mu(pa+p^n\Bbb Z_p)=\dfrac{1}{p^{n-1}}$.    We denote
  $\mu_m=\mu^m$ which is the product measure on $(p\Bbb{Z}_p)^m$. It is easy to show that (see Lemma \ref{lem:null})
$$\mu_m((p{\Bbb Z}_p)^m\backslash D_m')=0.$$
  Therefore, if $T$ is a transformation of $(p\Bbb{Z}_p)^m$ such that
$D_m'$ is $T$-invariant, to study the ergodicity of $T$ with respect to $\mu$, we only need to study the restriction
of $T$ on $D_m'$.

\begin{defn}\label{d-1-1}
Let  $i\in\{1,\dots, m\}$, $\sigma$ be a permutation   on the set $\{1,2,\ldots,m\}$, and let
$\bp=(p_1,\dots, p_m)\in ({\Bbb Q}^{\times})^m, \bq=(q_1,\dots, q_m)\in \Bbb Q^m$.
We define  $F:~D_m\to D_m$ by
$F(x_1,\ldots,x_m)=(y_1,\ldots,y_m)$, where for $k=1,2,\ldots, m$,
$$
y_k=\begin{cases}{p_k}/{x_i}-q_k&\text{if $k=\sigma^{-1}(i)$},\\
{p_kx_{\sigma(k)}}/{x_i}-q_k &\text{if $k\ne\sigma^{-1}(i)$};\end{cases}$$
We call $F$  the \it{$m$-dimensional linear fractional    transformation (LFT)  with parameter} $(i,\sigma, \bp,\bq)$.
\end{defn}

\begin{rem}\label{rem:range}  Notice that
 $1,x_1,\ldots,x_m$ are linearly independent over ${\Bbb Z}$ implies that
 $1/x_i, $ $x_1/x_i,$ $\ldots,x_m/x_i$  also do, so $F(x_1,\ldots,x_m)$ defined above belongs to $D_m$.
\end{rem}

\begin{lem}\label{l-2}
Let  $F$ be an $m$-dimensional LFT  with parameter $(i,\sigma,\bp,\bq)$.
Then $F$  is a bijection on  $D_m$; precisely,
 if we denote \\$(x_1,\dots,x_m)=F^{-1}(y_1,\ldots,y_m)$, then
\begin{equation}\label{eq:x_k}
x_k=\begin{cases}\dfrac{p_s}{y_s+q_s}&\text{if $k=i$},\\
 \dfrac{p_s(y_{t}+q_t)}{p_t(y_s+q_s)} &\text{if $k\ne i$},\end{cases}
\end{equation}
where $s=\sigma^{-1}(i),t=\sigma^{-1}(k)$.
\end{lem}
\begin{proof}
 Since $y_s$ is not a rational number, so $y_s+q_s\neq0$,  and \eqref{eq:x_k} holds by a easy calculation.
The lemma is proved.
\end{proof}

We are interested in the $m$-dimensional LFT $F$  satisfying $F^{-1}(D_m')\subset D_m'$.
To this end, we pose the following  conditions.

From now on, since $p$ is fixed, we will abbreviate $ord_p(x)$
by $ord(x)$.

\begin{defn}\label{d-1}
Let $F$ be  an $m$-dimensional LFT  with parameter $(i,\sigma, \bp,\bq)$.
 Denote $s=\sigma^{-1}(i)$. We say $F$ is \textit{hyperbolic}, if   the following conditions hold:
\begin{enumerate}
\item[(i)] For every $k\in \{1,\dots, m\}$,  $p_k\ne 0$ and $ord(p_k)\geq  0$;

\item[(ii)] $ord(q_s)\leq  0$ and $ord(p_s/q_s)>0$ (note that $q_s\neq 0$);

\item[(iii)] If  $k\neq s$ and  $ord(q_k)\leq  0$,  then  $ord(p_s/q_s)-ord(p_k/q_k)>0$;

\item[(iv)] If  $k\neq s$ and  $ord(q_k)> 0$,  then  $ord(p_s/q_s)-ord(p_k)>0$.
\end{enumerate}
\end{defn}

(If $q_k=0$, we make the convention that $ord(p_k/q_k)=-\infty$.) Roughly speaking, if $p_k$ are small and $q_k$ are big in the $p$-norm, then $F$ is hyperbolic.
The following is a sufficient condition for a LFT to be hyperbolic.

\begin{lem}\label{l-2-2}
Let  $F$ be an  $m$-dimensional LFT  with parameter $(i,\sigma, \bp,\bq)$ with
$ord(q_s)\leq  0$ where $s=\sigma^{-1}(i)$.
Suppose that for every $k\in \{1,\dots, m\}$,
 $p_k\ne 0$, $ord(p_k)\geq  0$,
and if  $ord(q_k) > 0$, then  $ord(p_k)=0$.
If there exists   $\bar x\in D_m'$
such that
$F(\bar x )\in D_m'$,
then  $F$ is   hyperbolic.
\end{lem}

\begin{proof}
  Clearly $F$ satisfies the conditions (i); moreover, under our assumption,   (ii) implies (iv) in Definition \ref{d-1}.
We will show that $F$ satisfies the conditions (ii) and (iii).
Write $\bar x=(\bar x_1,\dots, \bar x_m)$ and
$(\bar y_1,\dots, \bar y_m)=F(\bar x)$, then  $(\bar y_1,\dots, \bar y_m)\in D_m'$.

Take $k\ne i$, then by Lemma \ref{l-2},
$\bar x_k=\dfrac{p_s(\bar y_{t}+q_t)}{p_t(\bar y_s+q_s)}$
where $s=\sigma^{-1}(i),t=\sigma^{-1}(k)$.
By considering the valuations of both sides for the case of $ord(q_t)\leq 0$ we have
$$ord(p_s)-ord(q_s)+ord(q_t)-ord(p_t)=ord(\bar x_k)>0,$$
which verifies condition (iii). If $ord(q_t)>0$, we get condition (iii) easily. 
Take $k=i$,  a similar argument shows that condition (ii) is fulfilled. The lemma is proved.
\end{proof}

\begin{example}  Let $F$ be a hyperbolic one dimensional LFT with parameter $(i,\sigma,\bp,\bq)$, then
 we have  $i=1$, $\sigma$ is the identity map, and
 $(\bp,\bq)=(p_1,q_1)$. Therefore
 $$
 F(x)=\frac{p_1}{x}-q_1.
 $$
A $p$-adic continued fractional algorithm is usually related to a countable family of such transformations,
which is the  concern of the next section.
\end{example}

By Lemma \ref{l-2}, $F^{-1}$ is well-defined on $D_m'$ by   \eqref{eq:x_k}.
If $F$ is a hyperbolic $m$-dimensional LFT,   we  can extend $F^{-1}$ to $(p{\Bbb Z}_p)^m$ by the same formula  \eqref{eq:x_k}.

\begin{lem}\label{l0}
Let  $F$ be a hyperbolic $m$-dimensional LFT with parameter $(i,\sigma,\bp,\bq)$. Then
   $F^{-1}((p{\Bbb Z}_p)^m)\subset (p{\Bbb Z}_p)^m$ and  $F^{-1}(D_m')\subset D_m'$.
\end{lem}

\begin{proof}
Let     $(y_1,\ldots,y_m)\in  (p{\Bbb Z}_p)^m$  and   denote  $(x_1,\ldots,x_m)=F^{-1}(y_1,\ldots,y_m)$.
Then, $x_i=\dfrac{p_s}{y_s+q_s}$, where  $s=\sigma^{-1}(i)$, so
 $ord(x_i)=ord(p_s)-ord(q_s)>0.$
 Pick $k\ne i$  and denote $t=\sigma^{-1}(k)$,  then   $x_k=\dfrac{p_s(y_{t}+q_t)}{p_t(y_s+q_s)}$.
 If $ord(q_t)\leq 0$, then
$$ord(x_k)=ord(p_s)-ord(q_s)+ord(q_t)-ord(p_t)>0;$$
if $ord(q_t)>0$, then
$$ord(x_k)>ord(p_s)-ord(q_s)-ord(p_t)>0.$$
Therefore, $(x_1,\ldots,x_m)\in  (p{\Bbb Z}_p)^m$, which proves the first assertion.
The second assertion holds
 since that $1,y_1,\dots, y_m$ are
linearly independent over $\mathbb Z$ implies that $1,x_1,\dots, x_m$ also do.
\end{proof}

From now on, we will always regard $F^{-1}$ as a transformation on $(p{\Bbb Z}_p)^m$.
For $x=(x_1,\dots, x_m)\in (p{\Bbb Z}_p)^m$, we define $|x|_p=\max_{1\leq j\leq m}|x_j|$.

\begin{lem}[\textbf{Hyperbolicity Lemma}]\label{l01}
Let  $F$ be a hyperbolic $m$-dimensional LFT  with parameter $(i,\sigma,\bp,\bq)$.
Then,  for $x,y\in (p{\Bbb Z}_p)^m$, it holds that
  $$|F^{-1}(x)-F^{-1}(y)|_p\leq p^{-1}|x-y|_p.$$
\end{lem}
\begin{proof}
Denote   $x=(a_1,\ldots,a_m), y=(b_1,\ldots,b_m)$,
  $(c_1,\ldots,c_m)=F^{-1}(x)$ and  $(d_1,\ldots,d_m)=F^{-1}(y)$.

  For $k=i$, we have
\begin{align*}
c_i-d_i=\dfrac{p_s}{a_s+q_s}-\dfrac{p_s}{b_s+q_s}=\dfrac{p_s(b_s-a_s)}{(a_s+q_s)(b_s+q_s)},
\end{align*}
where  $s=\sigma^{-1}(i)$.
By Definition \ref{d-1} (ii) we have
\begin{equation}\label{eq:case-i}
|c_i-d_i|_p<|b_s-a_s|_p.
\end{equation}

For $k\ne i$, we have
\begin{align*}
&c_k-d_k=\dfrac{p_s(a_{t}+q_t)}{p_t(a_s+q_s)}
-\dfrac{p_s(b_{t}+q_t)}{p_t(b_s+q_s)}\\
&=\dfrac{p_s(q_s(a_{t}-b_{t})+q_t(a_{s}-b_{s})+b_{s}(a_{t}-b_{t})-b_{t}(a_{s}-b_{s}))}{p_t(a_s+q_s)(b_s+q_s)},
\end{align*}
where  $s=\sigma^{-1}(i)$ and $t=\sigma^{-1}(k)$.
Therefore, we have
\begin{align*}
&ord(c_k-d_k)\geq ord(p_s)-ord(p_t)-2ord(q_s)\\
&+\min(ord(q_s),ord(q_t),ord(b_s),ord(b_{t}))+\min(ord(a_{t}-b_{t}),ord(a_{s}-b_{s})).
\end{align*}
By Definition \ref{d-1} (iii) (iv) we have
\begin{align*}
ord(p_s)-ord(p_t)-2ord(q_s)+\min(ord(q_s),ord(q_t),ord(b_s),ord(b_{t}))>0,
\end{align*}
which implies
\begin{equation}\label{eq:case-k}
|c_k-d_k|_p<\max\{|b_s-a_s|_p,|b_t-a_t|_p\}.
\end{equation}
The lemma is verified by \eqref{eq:case-i} and \eqref{eq:case-k}.
\end{proof}

\section{Continued fractional system}

In this section, we introduce the notion of continued fractional system. Several
 important families of such systems are constructed.

\begin{defn}\label{d-2}
Let $\Lambda$ be a countable set, and let $\{F_{\lambda} \}_{\lambda \in \Lambda}$
be a family of $m$-dimensional hyperbolic LFTs.
We say  that $\{F_{\lambda} \}_{\lambda \in \Lambda}$ is a $m$-\emph{dimensional  continued fraction system}, if
 $\{F_{\lambda}^{-1}(D'_m)\mid \lambda \in \Lambda\}$ is a partition of $D_m'$.
\end{defn}

\begin{defn}\label{d-2-2}
Let  $\{F_{\lambda} \}_{\lambda \in \Lambda}$ be a  hyperbolic $m$-dimensional continued fraction system.
We define a transformation  $T$ on $D'_m$ associated with  $\{F_{\lambda} \}_{\lambda \in \Lambda}$ by
$$
T(\alpha):=F_{\lambda}(\alpha),\quad \text{if }  \alpha \in F_{\lambda}^{-1}(D'_m).
$$
\end{defn}


For  $\alpha \in D'_m$,  we define
\begin{equation}
\psi(\alpha):=\lambda,  \quad \text{ if }  \alpha \in F_{\lambda}^{-1}(D'_m).
\end{equation}
We say that $(D'_m,T,\psi)$ is the \emph{continued fraction algorithm} associated with
 the system $\{F_{\lambda} \}_{\lambda \in \Lambda}$.
For $j\in {\Bbb Z}_{\geq 0}$, we define  $\psi_j:=\psi(T^j(\alpha))$.
We define the $j$-$th$ $convergent$ $\pi(\alpha;j)$ of $\alpha$ by
$$
\pi(\alpha;j):=(F_{\psi_0}^{-1} \cdots  F_{\psi_j}^{-1})(\bar{0}),
$$
where $\bar{0}=(0,\ldots,0)\in (p{\Bbb Z}_p)^m$.



 \begin{thm}\label{t0}
  For any $\alpha\in D_m'$, we have
$\displaystyle \lim_{j\to \infty}\pi(\alpha;j)=\alpha$.
\end{thm}
\begin{proof}
Let $\alpha\in D_m'$.
Applying  Lemma \ref{l01} inductively we have for $j\in {\Bbb Z}_{\geq 0}$.
\begin{align*}
&|\alpha-\pi(\alpha;j)|_p=|\alpha-(F_{\psi_0}^{-1} \cdots  F_{\psi_j}^{-1})(\bar{0})|_p\\
&\leq \dfrac{1}{p^j}|(F_{\psi_j} \cdots  F_{\psi_0})(\alpha)-\bar{0}|_p
\leq \dfrac{1}{p^j},
\end{align*}
which implies $\displaystyle \lim_{j\to \infty}\pi(\alpha;j)=\alpha$.
\end{proof}

\subsection{One dimensional  systems}
We start with some notations.
\begin{defn}

We define
\begin{align*}
J_N=\left \{\sum_{i=-N}^0 c_i p^i \in \Bbb{Q}_p|c_i\in \{0,1,\ldots,p-1\}, c_{-N}\neq 0 \right \}
\text{ and } J=\bigcup_{N\geq 0} J_N.
\end{align*}
By convention, we set $J_{\infty}=\emptyset$ and $J_N=\{0\}$ if $N<0$. Clearly $J$ is a subset of $\Bbb{Q}$.
\end{defn}

Recall that a one-dimensional LFT must have the form
 $$
 F(x)=\frac{p_1}{x}-q_1
 $$
 where $p_1,q_1\in\Bbb Q$.
We are especially interested in the  following special case:  for $k{\geq 0}$ and $v\in J\setminus \{0\}$, we denote
\begin{equation}
F_{k,v}=\frac{p^k}{x}-v.
\end{equation}
We see easily that
$F_{k,v}$ is   hyperbolic  if and only if
$k>0$ or   $ord(v)<0$.

\begin{defn}
Let  $\ell \in {\Bbb Z}_{\geq 0}\cup \{\infty\}$. We define  $T_\ell:~ p\Bbb{Z}_p\to p\Bbb{Z}_p$ as
$$
T_\ell(x):=\dfrac{p^k}{x}-\left\lfloor \dfrac{p^k}{x} \right\rfloor_p,
$$
where $k=(ord(x)-\ell)\vee 0$.  (If $x=0$, we set $T_\ell(0)=0$.)
\end{defn}

Then, $T_0$ is the transformation associated with Schneider continued fraction algorithm,
and $T_\infty$ is the transformation associated with Ruban continued fraction algorithm
(Indeed,  $T_{\infty}(x)=\dfrac{1}{x}-\left\lfloor \dfrac{1}{x} \right\rfloor_p$).

\begin{thm}\label{l3-1}
Let  $\ell\in {\Bbb Z}_{\geq 0}\cup \{\infty\}$.
Then $T_\ell$ is the transformation on $D_1'$ associated with the continued fraction system
\begin{equation}\label{eq:F_n}
\{F_{k,v}\mid {k{>0}, v\in J_\ell}\}\cup \{F_{0,v}\mid {v\in J_1\cup \cdots\cup J_\ell} \}.
\end{equation}
\end{thm}

We remark that if $\ell=0$, then the second term in \eqref{eq:F_n} is an empty set; if $\ell=\infty$,
then the first term in \eqref{eq:F_n} is the empty set.

\begin{proof} Fix  $\ell\in {\Bbb Z}_{\geq 0}$.
We see easily that every  LFT appeared in  \eqref{eq:F_n}  is hyperbolic.

Pick $x\in D_1'$. Let $k=(ord(x)-\ell)\vee 0$ and set $v=\left\lfloor \dfrac{p^{k}}{x} \right\rfloor_p$, then
$$
T_\ell(x)=\dfrac{p^{k}}{x}-v=F_{k,v}(x)\in D_1'.
$$

If $k>0$, we see that $ord(x)=\ell+k$,
 which implies  that
$ v\in J_\ell\setminus\{0\}$;
if $k=0$,  then  $1\leq ord(x)\leq \ell$, so
$v\in J_j$ for some $1\leq j\leq \ell$. Therefore, we have
\begin{align}\label{D_1}
 D_1'=\bigcup_{k\ge 1, v\in J_\ell} F_{k,v}^{-1}(D_1')\cup \bigcup_{1\leq k\leq \ell,v\in J_k} F_{0,v}^{-1}(D_1').
\end{align}
To prove the disjointness of the right hand side of (\ref{D_1}), notice that if the fractional parts
of $p^k/x$ and $p^{k'}/x$ coincide and $k\neq k'$, then the $p$-adic expansion of $x$ is eventually periodic
and hence $x$ is a rational number, which contradicts $x\in D'_1$.
This verifies the theorem for $\ell\in \Bbb{Z}_{\geq 0}$.

Let  $l=\infty$.
In this case $T_\infty(x)=\dfrac{1}{x}-\left\lfloor \dfrac{1}{x} \right\rfloor_p=F_{0,v}(x)$ where $v=\left\lfloor \dfrac{1}{x} \right\rfloor_p\in J_{ord(x)}$.
Therefore, we have
\begin{align}\label{D_2}
D_1'=\bigcup_{k\geq 1}\bigcup_{v\in J_k} F_{0,v}^{-1}(D_1').
\end{align}
The disjointness of the right hand side of (\ref{D_2}) is obvious.   This verifies the case $\ell=\infty$ and completes the proof of the theorem.
\end{proof}

Recall that $\mu$ is the Haar measure on $\mathbb Q_p$. As a consequence of  Theorem \ref{t1} and \ref{t3}, we have

\begin{cor}\label{t4}
The measure $\mu$ is an ergodic measure of the transformation $T_\ell$,  $\ell\in {\Bbb Z_{\geq 0}}\cup \{\infty\}$.
\end{cor}

 Let  $\ell\in {\Bbb Z}_{\geq 0}\cup \{\infty\}$.
We define  two functions $a_{\ell}$ and $b_{\ell}$ on $D_1'$  as follows: for $x\in D_1'$, 
\begin{align*}
(b_{\ell}(x), a_{\ell}(x)):=\begin{cases}
(k,v),& \ \text{ if }x\in F_{k,v}^{-1}(D_1') \text{ with }k\ge 1, v\in J_\ell,\\
(0,v),& \ \text{ if }x\in F_{0,v}^{-1}(D_1') \text{ with } v\in \bigcup_{k=1}^\ell J_k.
\end{cases}
\end{align*}
The functions $a_{\ell}$ and $b_{\ell}$ are  well defined by the formula (\ref{D_1}).
  Then,  by the continued fraction system (\ref{eq:F_n}), a point 
$\xi\in D_1'$ has the following $p$-adic continued fractional expansion:
\begin{align*}
\xi=\cfrac{p^{b_1}}{a_1+\cfrac{p^{b_2}}{a_2+\cfrac{p^{b_3}}{a_3+\ldots}}}.
\end{align*}
 
As a straightforward application of Corollary \ref{t4} we have
\begin{thm}\label{t4a}
For almost everywhere $\xi\in D_1'$
\begin{align*}
&\lim_{n\to \infty}\dfrac{a_1+\ldots +a_n}{n}=\dfrac{p}{2},
\end{align*}
and
\begin{align*}
&\lim_{n\to \infty}\dfrac{b_1+\ldots +b_n}{n}=\dfrac{p}{p^{\ell}(p-1)}.
\end{align*}
\end{thm}

\begin{rem}
In the case of $\ell=0$ in Theorem \ref{t4a}
the former formula  was 
shown by Han\v{c}l, Ja\v{s}\v{s}ov\'a, Lertchoosakul and Nair \cite{HJLN}
and the later was
shown by Hirsh and Washington \cite{HW}.
\end{rem}

\begin{proof} 
Since $a_{\ell}\in L_1(D_1',\mu)$, by Birkhoff's ergodic Theorem, for almost everywhere $\xi\in D_1'$ it holds that
\begin{align*}
&\lim_{n\to \infty}\dfrac{a_1+\ldots +a_n}{n}=
\lim_{n\to \infty}\dfrac{a_{\ell}(\xi)+\ldots +a_{\ell}(T_{\ell}^{n-1}\xi)}{n}\\
&=\int_{D_1'} a_{\ell}(\xi)d\mu.
\end{align*}
By (\ref{D_1}) and Lemma \ref{l4}  we have
\begin{align*}
&\int_{D_1'} a_{\ell}(\xi)d\mu=\sum_{1\leq k\leq \ell,v\in J_k}v\mu(F_{0,v}^{-1}(D_1'))+
\sum_{k\ge 1, v\in J_\ell}v\mu(F_{k,v}^{-1}(D_1'))
\\
&=\sum_{1\leq k\leq \ell,v\in J_k}v\iota(F_{0,v})^{-1}+
\sum_{k\ge 1, v\in J_\ell}v\iota(F_{k,v})^{-1}
\\
&=\sum_{1\leq k\leq \ell,v\in J_k}\dfrac{v}{p^{2k}}+
\sum_{k\ge 1, v\in J_\ell}\dfrac{v}{p^{k+2\ell}}
\\
&=\dfrac{p}{2}.
\end{align*}
Similarly, for almost everywhere $\xi\in D_1'$ we have
\begin{align*}
&\lim_{n\to \infty}\dfrac{b_1+\ldots +b_n}{n}=\int_{D_1'} b_{\ell}(\xi)d\mu\\
&=\sum_{k\ge 1, v\in J_\ell}k\iota(F_{k,v})^{-1}\\
&=\sum_{k\ge 1, v\in J_\ell}\dfrac{k}{p^{k+2\ell}}\\
&=\dfrac{p}{p^{\ell}(p-1)}.
\end{align*}
\end{proof}

\subsection{Multi-dimensional   systems}\label{Multi}
Let $m\in {\Bbb Z_{> 0}}$ and $\ell\in {\Bbb Z_{\geq 0}}\cup \{\infty\}$.
We define a transformation $T_{\ell,m}$ on $D_m'$ as follows.

If $m=1$, we set $T_{\ell,m}:=T_\ell$.

Let $m>1$.
For $x=(x_1,\ldots,x_m)\in D_m'$, we define $T_{\ell,m}(x):=(y_1,\ldots,y_m)$ as
\begin{align*}
y_k:=\begin{cases}
T_\ell (x_1), &\text{ if  $k=m$,}\\
\dfrac{p^{r_k}x_{k+1}}{x_{1}}-\left\lfloor \dfrac{p^{r_k}x_{k+1}}{x_{1}} \right\rfloor_p,  &\text{ if $k\ne m$},
\end{cases}
\end{align*}
where $r_k=(-\ell-ord(x_{k+1})+ord(x_1))\vee 0$.

\begin{rem}
If $\ell=\infty$ and $m=2$, we have
\begin{align*}
T_{\infty,2}(x,y)=\left(\dfrac{y}{x}-\left\lfloor \dfrac{y}{x} \right\rfloor_p,\dfrac{1}{x}-\left\lfloor \dfrac{1}{x} \right\rfloor_p\right).
\end{align*}
This transformation is first considered by Ruban \cite{R2,R3}, as a $p$-adic version of Jacobi-Perron algorithm.
We remark that in general,  $T_{\infty,m}$ is a $p$-adic version of higher dimensional classical Jacobi-Perron algorithm.
\end{rem}

\begin{defn}\label{letl-1-0}
Let $\ell\in {\Bbb Z_{\geq 0}}$ and pick $x=(x_1,\ldots,x_m)\in D_m'$.
We define
$\bp^{(\ell,x)} \in \Bbb{Q}^m$   by
\begin{align*}\displaystyle
p^{(\ell,x)}_{k}:=\begin{cases}
p^{(-\ell+ord(x_1))\vee 0}, &\text{ if $k=m$},\\
p^{(-\ell-ord(x_{k+1})+ord(x_1))\vee 0}, &\text{ if $k\ne m$}.
\end{cases}
\end{align*}
For $\ell=\infty$ we define $\bp^{\infty,x}:=(1,\dots,1)$.

For $\ell\in {\Bbb Z_{\geq 0}}\cup \{\infty\}$,
  we define $\bq^{(\ell,x)}\in \Bbb{Q}^m$   by
\begin{align*}
q^{(\ell,x)}_{k}:=\begin{cases}
\left\lfloor \dfrac{p^{(\ell,x)}_{k}}{x_{1}} \right\rfloor_p, &\text{ if $k=m$},\\
\left\lfloor \dfrac{p^{(\ell,x)}_{k}x_{k+1}}{x_{1}} \right\rfloor_p, &\text{ if $k\ne m$}.
\end{cases}
\end{align*}
\end{defn}

Let $\sigma'$ be the permutation on $\{1,\dots, m\}$ defined by
$\sigma'(k)=k+1 \pmod{m}.$

\begin{defn}\label{letl-1-1}
Let $\ell\in {\Bbb Z_{\geq 0}}\cup \{\infty\}$.
For $x\in D_m'$
we define $F^{(\ell,x)}$ to be
 the $m$-dimensional LFT  with parameter $(1,\sigma', \bp^{(\ell,x)}, \bq^{(\ell,x)})$.
\end{defn}

\begin{rem}
We remark that
for $x\in D_m'$, $F^{(\ell,x)}(x)=T_{\ell,m}(x)$.
\end{rem}

\begin{lem}\label{l7}
For any  $\ell\in {\Bbb Z}_{\geq 0}\cup \{\infty\}$ and $x\in D_m'$,
the transformation $F^{(\ell,x)}$ is   hyperbolic.
\end{lem}

\begin{proof} Clearly $\bp^{(\ell,x)}$ and $\bq^{(\ell,x)}$ satisfy conditions in Lemma \ref{l-2-2}. Moreover, $F^{(\ell,x)}(x)=T_{\ell,m}(x)\in D_m'$, hence $F^{(\ell,x)}$ is   hyperbolic
by Lemma \ref{l-2-2}.
\end{proof}

Let $\ell\in {\Bbb Z_{\geq 0}}\cup \{\infty\}$.
For $x,y\in D_m'$
if $F^{(\ell,x)}=F^{(\ell,y)}$, we define $x\sim y$.
Clearly, this is an equivalence relation. We shall use the notation $F^{(\ell,[x])}$ where
$[x]$ is the equivalence class containing $x$.

\begin{thm}\label{l8}
Let  $\ell \in {\Bbb Z}_{\geq 0}\cup \{\infty\}$.
Then the family
$\{F^{(\ell,x)}\}_{x\in D_m'/\sim}$ is an
$m$-dimensional  continued fraction system.
\end{thm}
\begin{proof} For simplicity, in this proof, we denote $F^x:=F^{(\ell,x)}$, $\bp^x=:\bp^{(\ell,x)}$ and
$\bq^x=:\bq^{(\ell,x)}$.
By Lemma \ref{l7}, we only need to show that
$$
D_m'=\bigcup_{[x]\in D_m'/\sim }(F^{(\ell,[x])})^{-1}(D_m')
$$
is a partition. Suppose on the contrary that $[x]\neq [y]$ but
$$
F^x(z)=F^y(z):=(w_1,\dots, w_k)
$$
for some $z\in D_m'$. This implies that (see Definition \ref{d-1-1})
$$
w_1=\frac{p^x_1z_2}{z_1}-q^x_1=\frac{p^y_1z_2}{z_1}-q^y_1.
$$
If $p^x_1\neq p^y_1$, then  $z_2/z_1$ is eventually periodic in $p$-adic expansion, so $z_2/z_1\in {\mathbb Q}$, a contradiction. So we have $p^x_1=p^y_1$ and consequently $q^x_1=q^y_1$. For the other indices other than $1$, the same argument can be applied. Therefore, $\bp^x=\bp^y$ and $\bq^x=\bq^y$, so $x\sim y$. This contradiction proves the lemma.
\end{proof}

By Theorem \ref{t3} we have

\begin{cor}\label{t5}
The measure $\mu_m$ is an ergodic measure of the transform $T_{\ell,m}$.
\end{cor}

Inoue and Nakada \cite{IN} considered a modification of Brun's algorithm (\cite{Sc}) over a set of formal power series.
We give a $p$-adic version of Brun's algorithm as follows.

For $x=(x_1,\ldots,x_m)\in D_m'$, we define $T_{B}(x):=(y_1,\ldots,y_m)$ as
\begin{align*}
y_k:=\begin{cases}
\dfrac{1}{x_i}-\left\lfloor\dfrac{1}{x_i} \right\rfloor_p, &\text{ if  $k=i$,}\\
\dfrac{x_{k}}{x_{i}}-\left\lfloor\dfrac{x_{k}}{x_i} \right\rfloor_p,  &\text{ if $k\ne i$},
\end{cases}
\end{align*}
where $i=\min\{j\mid |x_j|_p=|x|_p\}$.
We also have the ergodicity for this algorithm  in the similar manner.

\section{Lemmas}
In this section, we prove several lemmas.
In what follows, we always use $a$ and $b$ to denote elements in $\Bbb{Z}_p$ and assume that $n\in {\Bbb Z_{>0}}$.
For sets $X,Y\subset {\Bbb Q}_p$ with $0\notin Y$ we define
$\dfrac{1}{Y}:=\left \{\dfrac{1}{y}|~y\in Y\right \}$
and $\dfrac{X}{Y}:=\left \{\dfrac{x}{y}|~x\in X,y\in Y\right \}$.

\begin{lem}\label{l1}
Let $k\in {\Bbb Z_{\geq 0}}$, $n\in {\Bbb Z_{>0}}$ and $a\in {\Bbb Z}_p$.  Then,
$$\dfrac{1}{pa+p^n{\Bbb Z}_p+\dfrac{1}{p^k}}=
p^k\left(\dfrac{1}{1+p^kpa}+p^{n+k}{\Bbb Z}_p\right)$$.
\end{lem}
\begin{proof}
Let $u\in p^n{\Bbb Z}_p$.
We have
\begin{align*}
&\dfrac{1}{pa+u+\dfrac{1}{p^k}}=\dfrac{p^k}{p^k(pa+u)+1}=
p^k\sum_{j=0}^{\infty}(-p^k(pa+u))^j
\\
&\in p^k\left(\dfrac{1}{1+p^kpa}+p^{n+k}{\Bbb Z}_p\right).
\end{align*}
Conversely, let $u\in p^n{\Bbb Z}_p$.
We have
\begin{align*}
&p^k\left(\dfrac{1}{1+p^kpa}+p^{k}u\right)=\dfrac{p^k}{\dfrac{1+p^kpa}{1+p^{k}u(1+p^kpa)}}\\
&=\dfrac{p^k}{(1+p^kpa)\sum_{j=0}^{\infty}(-p^ku(1+p^kpa))^j}\in \dfrac{1}{pa+p^n{\Bbb Z}_p+\dfrac{1}{p^k}}.
\end{align*}

\end{proof}

\begin{lem}\label{l2}
Let $k\in {\Bbb Z_{\geq 0}}$, $n\in {\Bbb Z_{>0}}$ and $a\in {\Bbb Z}_p$.
Let $v\in {\Bbb Q}$ such that $ord(v)=-k$. Then,
$$\dfrac{1}{pa+p^n{\Bbb Z}_p+v}=\dfrac{1}{v+pa}+p^{n+2k}{\Bbb Z}_p.$$
\end{lem}
\begin{proof}
Since $vp^k\in {\Bbb Z}_p\setminus p{\Bbb Z}_p$,  by Lemma \ref{l1}, we have
\begin{align*}
&\dfrac{1}{pa+p^n{\Bbb Z}_p+v}=\dfrac{1}{pa+p^n{\Bbb Z}_p+\dfrac{vp^k}{p^k}}=\dfrac{\dfrac{1}{vp^k}}{\dfrac{pa}{vp^k}+p^n{\Bbb Z}_p+\dfrac{1}{p^k}}\\
&=\dfrac{1}{v}\left(\dfrac{1}{1+p^k\dfrac{pa}{vp^k}}+p^{n+k}{\Bbb Z}_p\right)
=\dfrac{1}{v+pa}+p^{n+2k}{\Bbb Z}_p.
\end{align*}
\end{proof}

\begin{lem}\label{lem:null}
$\mu_m((p{\Bbb Z}_p)^m\backslash D_m')=0$.
\end{lem}
\begin{proof} Let us denote $\Omega_m':=(p{\Bbb Z}_p)^m\backslash D_m'$.
Since $\Omega_1'=p{\Bbb Z}_p\cap {\Bbb Q}$, we infer that every element in $\Omega_1'$
has a periodic $p$-adic expansion, so $\mu_1(\Omega_1')=0$.

Now we consider the case that $m\geq 2$.
Let $(\alpha_1,\dots, \alpha_m)\in \Bbb{Z}^m\setminus \{\bar 0\}$;
 let $j$ be the index such that  $\alpha_j\ne 0$.
We put
$$A_{\alpha_1,\ldots,\alpha_m}=\{(x_1,\ldots,x_m)\in (p\Bbb {Z}_p)^m|~\alpha_1x_1+\cdots+\alpha_mx_m=0\}.$$
First, we show $\mu_m(A_{\alpha_1,\ldots,\alpha_m})=0$.
Without loss of generality, we may assume that the common divisor of $\alpha_1,\dots, \alpha_m$ is $1$, and that $\alpha_m$ is not a multiple of $p$.
Then
$$x_m=(\alpha_1 x_1+\cdots \alpha_{m-1} x_{m-1})/\alpha_m$$
is a function on $(p\Bbb{Z})^{m-1}$. Clearly it is a continuous function, so its graph
$A_{\alpha_1,\dots, \alpha_m}$
is a compact subset of $(p\Bbb{Z})^{m}$ since $p\Bbb{Z}$ is compact. This proves that
$A_{\alpha_1,\dots, \alpha_m}$ is a closed set and hence it is measurable.
So by Fubini Theorem, the measure of $A_{\alpha_1,\dots, \alpha_m}$ is $0$.

Finally, since
$
\Omega_m'=\bigcup_{(j_1,\ldots, j_{m})\in {\Bbb Z}^m\backslash\{(0,\ldots,0)\}} A_{j_1,\ldots,j_m}\cap (p{\Bbb Z}_p)^m,
$
we  obtain the lemma.
\end{proof}

\section{Ergodicity}

In this section, we will always assume that  $F$ is a hyperbolic $m$-dimensional LFT  with parameter $(i,\sigma, \bp,\bq)$. Moreover, we will denote $s=\sigma^{-1}(i)$, and
\begin{equation}\label{eq:uv}
u=-ord(q_s),  \quad v=ord(p_s),
\end{equation}
\begin{equation}
h=\max\{ord(p_s)-ord(p_k)\mid k=1,\dots, m \}.
\end{equation}
  For $a=(a_1,\dots, a_m)\in {\Bbb Q}_p^m$, let us denote
  $$a+(p^n{\Bbb Z}_p)^m=(a_1+p^n{\Bbb Z}_p)\times\cdots \times (a_m+p^n{\Bbb Z}_p).$$

 Let $n\geq 1$ be an integer, let $a=(a_1,\ldots,a_m)\in (p{\Bbb Z}_p)^m$, let $y$ be an integer  such that $0\leq y < p^h$. We define cylinders of $p{\mathbb Z}_p$
 ($k=1,\dots, m$) by
$$V^{(y,n,a)}_k=\begin{cases}\dfrac{p_s}{a_s+q_s}+p^{n+v+2u}(y+p^{h}{\Bbb Z}_p),&\text{ if $k=i$},\\
\dfrac{1}{p_t}\left(\dfrac{p_s}{a_s+q_s}+p^{n+v+2u}y\right)\left(a_{t}+q_t+p^{n}{\Bbb Z}_p\right),
&\text{ if $k\ne i$},\end{cases}$$
where $t=\sigma^{-1}(k)$.
In case that $a$ and $n$ are fixed, we will simply denote $V^{(y,n,a)}$ by $V^{(y)}$.
We will use $\sqcup$ to denote a disjoint union.

\begin{lem}\label{l3}
Let $n\geq 1$ and let $a=(a_1,\ldots,a_m)\in (p{\Bbb Z}_p)^m$.
Then
\begin{align}\label{eq1}
F^{-1}(a+(p^n{\Bbb Z}_p)^m)
=\bigsqcup_{0\leq y <p^h}V_1^{(y)}\times \ldots \times V_m^{(y)}.
\end{align}
\end{lem}
\begin{proof}
First, we  show that the right hand side of   (\ref{eq1}) is a disjoint union.
Suppose that $V_i^{y_1}\cap V_i^{y_2}\ne \emptyset$ for $y_1,y_2\in \{0,1,\dots, p^h-1\}$.
Then, there exist $\gamma_1,\gamma_2\in {\Bbb Z}_p$ such that
$$\frac{p_s}{a_s+q_s}+p^{n+v+2u}(y_1+p^{h}\gamma_1)=\frac{p_s}{a_s+q_s}+p^{n+v+2u}(y_2+p^{h}\gamma_2),$$
which implies $y_1-y_2\in p^h{\Bbb Z}_p$. Therefore, $y_1=y_2$.

Next, we  show that the right side of   (\ref{eq1}) is included  in the left side of the equation.
Let $y$ be an integer with $0\leq y<p^h$.  We define a function
$$G:~({\Bbb Q}_p)^m \to ({\Bbb Q}_p)^m$$
 by
$(\tau_1,\dots,\tau_m)=G(\beta_1,\dots, \beta_m)$ where
\begin{align}
\tau_k=\begin{cases}\dfrac{p_s}{a_s+q_s}+p^{n+v+2u}(y+p^{h}\beta_k) &\text{if $k=i$},\\
\dfrac{1}{p_t}\left(\dfrac{p_s}{a_s+q_s}+p^{n+v+2u}{y}\right)\left(a_{t}+q_t+p^{n}\beta_k\right)
&\text{if $k\ne i$},\end{cases}\label{tauda}
\end{align}
 where $t=\sigma^{-1}(k)$. Then $  V_1^{(y)}\times \ldots \times V_m^{(y)}=G(({\Bbb Z}_p)^m)$.

 Pick  $(\beta_1,\ldots,\beta_m)\in {\Bbb Z}_p^m$ and denote
 $(\tau_1,\dots,\tau_m)=G(\beta_1,\dots, \beta_m)$.

By Lemma \ref{l2} we have
\begin{align}\label{eqp_sa_s}
\dfrac{p_s}{a_s+p^n{\Bbb Z}_p+q_s}=\dfrac{p_s}{a_s+q_s}+p_sp^{n+2u}{\Bbb Z}_p.
\end{align}
Therefore, there exists $\delta,\delta'\in {\Bbb Z}_p$ such that
\begin{align}
&\dfrac{p_s}{a_s+q_s}+p^{n+v+2u}y=\dfrac{p_s}{a_s+p^n\delta+q_s},\label{eqp_s}\\
&\dfrac{p_s}{a_s+q_s}+p^{n+v+2u}(y+p^{h}\beta_i)=\dfrac{p_s}{a_s+p^n\delta'+q_s}\label{eqp_s2}.
\end{align}

Let $u=(u_1,\dots, u_m)\in ({\Bbb Q}_p)^m$ such that $u_i=\delta'$.
Then $F^{-1}(a+p^nu)$ is well defined for such $u$.
 Solving the equation
\begin{equation}\label{eq:FG}
F^{-1}(a+p^n u)=G(\beta),
\end{equation}
we obtain that for  $t\ne s$, it holds that
\begin{align}\label{eqp_s3}
\dfrac{1}{a_s+p^n\delta+q_s}\left(a_{t}+q_t+p^{n}\beta_k\right)=
\dfrac{1}{a_s+p^n\delta'+q_s}\left(a_{t}+q_t+p^{n}u_{t}\right).
\end{align}
(Notice that for the $i$-th coordinate, equation \eqref{eq:FG} holds by \eqref{eqp_s2}.)

 We claim that $u=(u_1,\dots, u_m)\in ({\mathbb Z}_p)^m$.
By (\ref{eqp_s}), (\ref{eqp_s2}) and (\ref{eqp_s3}), for $t\neq s$, an easy calculation shows that
\begin{equation}\label{equ_k}
 u_{t}=-\dfrac{(a_{t}+q_t)p^{v+2u+h}(a_s+p^n\delta'+q_s)}{p_s}\beta_i+
 \dfrac{a_s+p^n\delta'+q_s}{a_s+p^n\delta+q_s}\beta_k.
\end{equation}
Let us denote the orders of the first term and the second term by $O_1$ and $O_2$, respectively. Then
\begin{align}
O_1
&\geq \min\{ord(q_t),ord(a_{t})\}+v+2u+h+ord(q_s)-ord(p_s)\nonumber\\
&=\min(ord(q_t),ord(a_{t}))-ord(q_s)+h. \label{pqt}
\end{align}
If $ord(q_t)> 0$, then
$$O_1> -ord(q_s)+ord(p_s)-ord(p_t)>0$$
 by Condition (iv) in Definition \ref{d-1};
if $ord(q_t)\leq 0$, then
$$O_1\geq ord(q_t)-ord(q_s)+ord(p_s)-ord(p_t)>0,$$
by Condition (iii) in Definition \ref{d-1}.
Therefore, we always have $O_1>0$.
We see easily
$
O_2=ord  ( \beta_k )\geq 0.
$
Therefore, we have $u_{t}\in {\Bbb Z}_p$. Our claim is proved and the desired inclusion relation is confirmed.

Finally, we will show that the left side of   (\ref{eq1}) is included  in the right side of the equation.
Let $u=(u_1,\ldots,u_m)\in {\Bbb Z}_p^m$.
By (\ref{eqp_sa_s}) there exists $\omega\in {\Bbb Z}_p$ such that
$$\dfrac{p_s}{a_s+p^nu_s+q_s}=\dfrac{p_s}{a_s+q_s}+p^{n+v+2u}\omega.$$
Let $y$ be the integer with $0\leq y<p^h$ such that $\omega-y\in p^h{\Bbb Z}_p$. We set
$\delta'=u_i$   and let $\delta\in {\mathbb Z}_p$ be the number defined by \eqref{eqp_s}.
Solving the equation $F^{-1}(a+p^nu)=G(\beta)$ where $\beta=(\beta_1,\dots, \beta_m)\in ({\Bbb Q}_p)^m$, we obtain
$$\beta_i=(\omega-y)/p^h,$$
 and for $k\ne i$,  equation \eqref{equ_k} still holds.
Since $u_t, \beta_i\in {\Bbb Z}_p$, and in \eqref{equ_k}, the order of the coefficient of $\beta_i$ is larger than $0$,
  we deduce that   $\beta_k\in {\Bbb Z}_p$.
Therefore, we have $F^{-1}(a+p^nu) \in V_1^{(y)}\times \ldots \times V_m^{(y)}$.
\end{proof}

For a set $A\subset (p{\Bbb Z}_p)^m$ we define $diam(A):=\max\{|x-y|_p\mid x,y\in A\}$.

\begin{lem}\label{l4}
Let  $n\in {\Bbb Z_{>0}}$ and denote  $t=\sigma^{-1}(k)$.
Then,
\begin{enumerate}
\item[(1)]
For $x=(x_1,\ldots,x_m)\in (p{\Bbb Z}_p)^m$,\\
$F^{-1}: x+(p^n{\Bbb Z}_p)^m  \to F^{-1}(x+(p^n{\Bbb Z}_p)^m)$ is a bijection.
\item[(2)]
$\text{diam}(F^{-1}(x+(p^n{\Bbb Z}_p)^m))\leq p^{-1}\text{diam}(x+(p^n{\Bbb Z}_p)^m)$.
\item[(3)]
$d\mu_m(F^{-1}(x))=\left ( {p^{mv+(m+1)u-\sum_{t\ne s}(ord(p_t))}} \right )^{-1}d\mu_m(x)$.
\end{enumerate}
\end{lem}
\begin{proof}
The proof of (1) is easy.
By Lemma \ref{l01} we have (2).

Now we prove (3).
Let $\alpha=(a_1,\ldots,a_m)\in (p{\Bbb Z}_p)^m$.
By Lemma \ref{l3},
\begin{align*}
F^{-1}(a+(p^n{\Bbb Z}_p)^m)
=\bigsqcup_{0\leq y <p^h}V_1^{(y)}\times \ldots \times V_m^{(y)},
\end{align*}
which implies
\begin{align*}
&\mu_m(F^{-1}(a+(p^n{\Bbb Z}_p)^m ))
=\sum_{0\leq y <p^h}\mu_m(V_1^{(y)}\times \ldots \times V_m^{(y)})\\
=&\sum_{0\leq y <p^h} \left (  {p^{n+v+2u+h-1+\sum_{t\ne s}(n+v+u-ord(p_t)-1)}} \right )^{-1}\\
=& \left ( {p^{mv+(m+1)u-\sum_{t\ne s}(ord(p_t))}} \right )^{-1}\mu_m( a+(p^n{\Bbb Z}_p)^m).
\end{align*}
It follows that $d\mu_m(F^{-1}(x))=\left ( {p^{mv+(m+1)u-\sum_{t\ne s}(ord(p_t))}}\right )^{-1}d\mu_m(x)$. Thus, we have the lemma.
\end{proof}

For  a hyperbolic $m$-dimensional LFT  $F$ with parameter $(i,\sigma, \bp,\bq)$,
we define
\begin{equation}
\iota(F):=p^{mv+(m+1)u-\sum_{t\ne s}(ord(p_t))}.
\end{equation}

\begin{thm}\label{t1}
Let  $\{F_{\lambda} \}_{\lambda \in \Lambda}$ be a  $m$-dimensional continued fraction system.
Let $T$ be the transformation  on $D'_m$ associated with  $\{F_{\lambda} \}_{\lambda \in \Lambda}$.
The measure $\mu_m$ is an invariant measure of $T$.
\end{thm}
\begin{proof}
Since $\displaystyle D'_m=\bigsqcup_{\lambda \in \Lambda}F_{\lambda}^{-1}(D'_m)$,
by Lemma \ref{l4} we have
\begin{align}\label{eq=1sum}
1=\sum_{\lambda \in \Lambda}\dfrac{1}{\iota(F_{\lambda})}.
\end{align}
Let $(a_1,\ldots,a_m)\in (p{\Bbb Z}_p)^m$ and $n\in {\Bbb Z}_{>0}$.
By Lemma \ref{l4}, we have
\begin{align*}
&\mu_m(T^{-1}(a+(p^n{\Bbb Z}_p)^m))
= \mu_m\left(\bigsqcup_{\lambda \in \Lambda}F_{\lambda}^{-1}( a+(p^n{\Bbb Z}_p)^m)\right)\\
=&\sum_{\lambda \in \Lambda}\mu_m(F_{\lambda}^{-1}( a+(p^n{\Bbb Z}_p)^m))
=\left(\sum_{\lambda \in \Lambda}\dfrac{1}{\iota(F)}\right)\mu_m( a+(p^n{\Bbb Z}_p)^m)\\
=&\mu_m( a+(p^n{\Bbb Z}_p)^m).
\end{align*}
Since $T$ is invariant on every cylinder, we obtain the theorem.
\end{proof}

A set $X\subset D_m'$ is said to be $T$-invariant if $X$ is measurable and $X=T^{-1}(X)$. We use $1_X$ to denote the characteristic function of $X$.

\begin{lem}\label{l5}
Let  $\{F_{\lambda} \}_{\lambda \in \Lambda}$ be an  $m$-dimensional continued fraction system.
Let $T$ be the transformation  on $D'_m$ associated with  $\{F_{\lambda} \}_{\lambda \in \Lambda}$ and $n\in {\Bbb Z}_{>0}$.
Let $a=(a_1,\ldots,a_m)\in D_m'$ and $A=(a+(p^n{\Bbb Z}_p)^m)\cap D_m'$.
 Let  $\lambda'\in \Lambda$ and
$B=F_{\lambda'}^{-1}(A)$. Let $X$ be an $T$-invariant subset of $D_m'$.
Then,
\begin{equation}\label{eq:AB}
 \dfrac{\mu_m(A\cap X)}{\mu_m(A)}=\dfrac{\mu_m(B\cap X)}{\mu_m(B)}.
 \end{equation}
Consequently, \eqref{eq:AB} holds for any Borel set $A\subset D_m'$ and $B=F_{\lambda'}^{-1}(A)$.

\end{lem}
\begin{proof} Since $X$ is $T$-invariant and $T^{-1}(X)=\sqcup_{\lambda\in \Lambda} F_\lambda^{-1}(X)$, we have
$$1_X(\alpha)=1_{T^{-1}X}(\alpha)=\sum_{\lambda\in \Lambda} 1_{F_\lambda^{-1}(X)}(\alpha).$$
Hence,
\begin{align*}
\mu_m(B\cap X)&=\int_B1_{X}(\beta)d\mu_m(\beta)\\
&=\int_A1_{X}(F_{\lambda'}^{-1}(\alpha))d\mu_m(F_{\lambda'}^{-1}(\alpha))\\
&=\int_A \sum_{\lambda\in \Lambda} 1_{F_\lambda^{-1}(X)}(F_{\lambda'}^{-1}(\alpha))d\mu_m(F_{\lambda'}^{-1}(\alpha))\\
&=\dfrac{1}{\iota(F_{\lambda'})}\int_A1_{X}(\alpha)d\mu_m(\alpha)\quad \quad \text{(By Lemma \ref{l4})}\\
&=\dfrac{1}{\iota(F_{\lambda'})}\mu_m(A\cap X).
\end{align*}
Again by  Lemma \ref{l4}, we have
\begin{align*}
\dfrac{\mu_m(A\cap X)}{\mu_m(B\cap X)}=\iota(F_{\lambda'})=\dfrac{\mu_m(A)}{\mu_m(B)},
\end{align*}
which implies the lemma.
\end{proof}

\begin{defn}
Let  $\{F_{\lambda} \}_{\lambda \in \Lambda}$ be a  $m$-dimensional continued fraction system.
For $\lambda_1,\ldots,\lambda_n \in \Lambda$ we define
$$\xi(\lambda_1,\ldots,\lambda_n):=F^{-1}_{\lambda_1}\cdots F^{-1}_{\lambda_n}(D_m').$$.
\end{defn}

\begin{lem}\label{l6}
Let  $\{F_{\lambda} \}_{\lambda \in \Lambda}$ be a  $m$-dimensional continued fraction system.
Then,
\begin{enumerate}
\item
For each $n\in {\Bbb Z}_{\geq 0}$
$\displaystyle D_m'=\bigsqcup_{\lambda_1,\ldots \lambda_n \in \Lambda}\xi(\lambda_1,\ldots,\lambda_n)$.
\item
For  $n\in {\Bbb Z}_{\geq 0}$ and $\lambda_1,\ldots, \lambda_n \in \Lambda$,
$$\text{diam}(\xi(\lambda_1,\ldots,\lambda_n))\leq \dfrac{1}{p^{n+1}}.$$

\item If $X$ is $T$-invariant, then
$$\frac{\mu_m(\xi(\lambda_1,\ldots,\lambda_n)\cap X)}{\mu_m(\xi(\lambda_1,\ldots,\lambda_n))}=\mu_m(X).$$
\end{enumerate}
\end{lem}

\begin{proof} Using Lemma \ref{l4} repeatedly,  we have (1) and (2).

Now we prove (3) by induction on $n$. For $n=1$, set $A=D_m'$ in Lemma \ref{l5}, we have
$$\frac{\mu_m(\xi(\lambda_1)\cap X)}{\mu_m(\xi(\lambda_1))}=\dfrac{\mu_m(X)}{\mu_m(D_m')}=\mu_m(X).$$
Hence, the conclusion holds since
$$ \frac{\mu_m(\xi(\lambda_1,\ldots,\lambda_n)\cap X)}{\mu_m(\xi(\lambda_1,\ldots,\lambda_n))}=\frac{\mu_m(\xi(\lambda_1,\ldots,\lambda_{n-1})\cap X)}{\mu_m(\xi(\lambda_1,\ldots,\lambda_{n-1}))}.$$
\end{proof}

\begin{thm}\label{t3}
Let  $\{F_{\lambda} \}_{\lambda \in \Lambda}$ be an $m$-dimensional continued fraction system.
Let $T$ be the transformation  on $D'_m$ associated with  $\{F_{\lambda} \}_{\lambda \in \Lambda}$.
The measure $\mu_m$ is an ergodic measure of the transformation $T$.
\end{thm}
\begin{proof}
Let $X$ be a $T$-invariant set.
Let $a=(a_1,\ldots,a_m)\in D_m'$, $k\in {\Bbb Z}_{>0}$ and
$U= (a+(p^k{\Bbb Z}_p)^m)\cap D_m'$.

By Lemma \ref{l6}(1),
\begin{equation}\label{part}
{\cal P}_n=\{\xi(\lambda_1,\dots, \lambda_n);~(\lambda_1,\dots, \lambda_n)\in \Lambda^n\}
\end{equation}
is a partition of $D_m'$. By Lemma \ref{l6}(2), if $n>k$, then the maximum of the diameters of the elements
in  ${\cal P}_n$  is smaller than $p^{-k}$.  So an element in ${\cal P}_n$ is either contained in $U$,
or is disjoint with $U$. Hence, there exists $H\subset \Lambda^n$ such that
$$
U=\bigsqcup_{(\lambda_1,\dots, \lambda_n)\in H}\xi(\lambda_1,\dots, \lambda_n).
$$
 By  Lemma \ref{l6}(3),
for every $(\lambda_1,\dots, \lambda_n)\in \Lambda^n$,
$$\mu_m(X)=\dfrac{\mu_m(\xi(\lambda_1,\ldots, \lambda_n)\cap X)}{\mu_m(\xi(\lambda_1,\ldots, \lambda_n)) } .$$
Therefore, we have
$$\dfrac{\mu_m(U\cap X)}{\mu_m(U)}=\mu_m(X).$$
Hence, by the density Theorem (see for instance, \cite{H}), for  $a.e.$ $x$,  $1_X(x)=\mu_m(X)$, which implies
$\mu_m(X)=0$ or $\mu_m(X)=1$.
We have the theorem.
\end{proof}

Finally, we show that  the systems are strong mixing.

\begin{thm}\label{ad2}
For any Borel subset $A,B$ in $D_m'$
\begin{align}\label{mixing}
\lim_{n\to \infty}\mu(T^{-n}A\cap B)=\mu(A)\mu(B).
\end{align}
\end{thm}

\begin{proof}
Let
\begin{align}
 \mathcal{P}=\{\xi(\lambda_1,\dots, \lambda_n)\mid n\in {\Bbb Z}_{\geq 0}, (\lambda_1,\dots, \lambda_n)\in \Lambda^n\}.
\end{align}
From Lemma \ref{l3} and \ref{l4}
we see that $\mathcal{P}$  generates the Borel algebra on $D_m'$.
Therefore, to prove the theorem, we only need to show that \eqref{mixing} holds for any $A, B\in {\mathcal P}$. 

Let $A=\xi(\tau_1,\dots, \tau_{n_1}), B=\xi(\phi_1,\dots, \phi_{n_2})$
where $n_1,n_2\in {\Bbb Z}$ and $\tau_i,\phi_j\in \Lambda$ for $1\leq i\leq n_1, 1\leq j\leq n_2$.
We suppose $n\geq n_2$.
Then,
\begin{align*}
&T^{-n}A\cap B\\
&=\left(\bigcup_{(\lambda_1,\dots, \lambda_n)\in \Lambda^n}
\xi(\lambda_1,\dots, \lambda_n,\tau_1,\dots, \tau_{n_1})\right)\cap\xi(\phi_1,\dots, \phi_{n_2})\\
&=\bigcup_{(\lambda_{n_2+1},\dots, \lambda_n)\in \Lambda^{n-n_2}}
\xi(\phi_1,\dots, \phi_{n_2},\lambda_{n_2+1},\dots,\lambda_n,\tau_1,\dots, \tau_{n_1}).\\
\end{align*}
Therefore by Lemma \ref{l4},  we have
\begin{align*}
&\mu(T^{-n}A\cap B)\\
=&\sum_{(\lambda_{n_2+1},\dots, \lambda_n)\in \Lambda^{n-n_2}} \left(\iota(F_{\phi_1})\cdots \iota(F_{\phi_{n_2}})
\iota(F_{\lambda_{n_2+1}})\cdots \iota(F_{\lambda_n})
\iota(F_{\tau_1})\cdots \iota(F_{\tau_{n_1}})\right)^{-1}\\
=&\left(\iota(F_{\phi_1})\cdots \iota(F_{\phi_{n_2}})\iota(F_{\tau_1})\cdots \iota(F_{\tau_{n_1}})\right)^{-1}\\
=&\mu(A)\mu(B),
\end{align*}
which completes the proof of the theorem.
\end{proof}


\noindent
{\bf Acknowledgements}

We thank Zhu Jiang for his earlier contribution on this subject.
We also thank the anonymous referee for valuable comments.

\noindent
Hui Rao: Department of Mathematics and Statistics, Central China Normal University, CHINA\\
{\it E-mail address: hrao@mail.ccnu.edu.cn}\\

\noindent
Shin-ichi Yasutomi: Faculty of Science, Toho University, JAPAN\\
{\it E-mail address: shinichi.yasutomi@sci.toho-u.ac.jp

\end{document}